%% file: SC_Gen_Lox_revised.tex
\documentclass[12pt,a4paper,reqno]{amsart}

\usepackage{amsmath, amsfonts, amssymb, amsthm, latexsym, nameref, enumerate} 
\usepackage{amsrefs} 
\usepackage{xspace}
\usepackage[T1]{fontenc}
\usepackage{graphicx, subfig}
\usepackage{hyperref}
\usepackage[usenames, dvipsnames]{xcolor}
\usepackage{tikz}
\usetikzlibrary{arrows,decorations.markings,decorations.pathreplacing,plotmarks}
\usepackage{float}
\usepackage{scrextend}
\usepackage{blindtext}
\setkomafont{labelinglabel}{\bfseries}

\newtheorem{proposition}{Proposition}[section]
\newtheorem{theorem}[proposition]{Theorem}
\newtheorem{corollary}[proposition]{Corollary}

\newtheorem{lemma}[proposition]{Lemma}

\theoremstyle{definition}
\newtheorem{definition}[proposition]{Definition}

\newtheorem{example}[proposition]{Example}
\newtheorem{remark}[proposition]{Remark}

\newcommand{\set}[1]{\left\{#1\right\}}
\newcommand{\setcon}[2]{\left\{#1\ \left|\ #2\right.\right\}}

\newcommand{\abs}[1]{\left\lvert#1\right\rvert}
\newcommand{\into}{\hookrightarrow}

\newcommand{\Z}{\mathbb{Z}}
\newcommand{\N}{\mathbb{N}}
\newcommand{\eps}{\varepsilon}

\newcommand{\diam}{\textup{diam}}
\newcommand{\Cay}{\textup{Cay}}

\newcommand{\sat}{\textup{sat}}

\newcommand{\fgen}[1]{\left\langle #1 \right\rangle}
\newcommand{\ngen}[1]{\left\langle\!\left\langle #1 \right\rangle\!\right\rangle}
\newcommand{\fpres}[2]{\left\langle #1 \left| #2 \right.\right\rangle}

\title[The geometry of generalized loxodromic elements]{The geometry of generalized loxodromic elements}

\author[Abbott]{Carolyn R. Abbott}
 \email{cra2112@columbia.edu}
\address{Department of Mathematics, Columbia University, 2990 Broadway, New York, NY 10027}

\author[Hume]{David Hume}
 \email{david.hume@maths.ox.ac.uk}
\address{Mathematical Institute, University of Oxford, Woodstock Road, Oxford OX2 6GG}

\date{\today}

\begin{document}

\begin{abstract}
\input{abstract}
\end{abstract}

\maketitle

\section{Introduction}\label{sec:intro}
\input{introduction_revised}

\section{Preliminaries}\label{sec:background}
\input{preliminaries_revised}

\section{Detecting generalized loxodromic elements}\label{section:loxodromics}
\input{firstconedoff_revised2}

\section{Detecting non-generalized loxodromic elements}

\input{nonGL_revised}

\bibliographystyle{alpha}
\bibliography{DB}

\end{document}

%% file: abstract.tex
We explore geometric conditions which ensure a given element of a finitely generated group is, or fails to be, generalized loxodromic; as part of this we prove a generalization of Sisto's result that every generalized loxodromic element is Morse. We provide a sufficient geometric condition for an element of a small cancellation group to be generalized loxodromic in terms of the defining relations and provide a number of constructions which prove that this condition is sharp.

%% file: introduction_revised.tex
A recurring theme in group theory is that a useful way to develop an understanding for a given group or class of groups satisfying some weak form of non-positive curvature is to construct interesting actions of them on hyperbolic spaces. Principal examples of this heuristic come from Bass-Serre theory, understanding actions of groups on trees in terms of amalgamations and HNN-extensions; understanding relatively hyperbolic groups via their actions on coned-off graphs, cusped spaces, or combinatorial horoballs; understanding mapping class groups via their actions on curve complexes of subsurfaces; and understanding $\operatorname{Out}(F_n)$ via its actions on the free factor and free splitting complex, among many others. One particular type of interesting 
 action which has received much recent attention is an acylindrical action.

An action of a group $G$ by isometries on a metric space $X$ is \textbf{acylindrical} if for all $d\geq0$ there exist constants $M,N\geq 0$ such that for all $x,y\in X$ with $d(x,y)\geq M$, the number of elements $g\in G$ satisfying $d(x,gx)\leq d$ and $d(y,gy)\leq d$ is at most $N$. An element $g$ of a group $G$ is \textbf{generalized loxodromic} if there is an acylindrical action of $G$ on a hyperbolic space $X$ such that $g$ acts loxodromically. A group is \textbf{acylindrically hyperbolic} if and only if it is not virtually cyclic and contains a generalized loxodromic element \cite{Os16}. 

The class of acylindrically hyperbolic groups is incredibly rich, including all non-elementary hyperbolic and relatively hyperbolic groups, as well as two of the most intensively studied classes of groups in recent years: mapping class groups and the outer automorphism groups of free groups. At the same time the consequences of being acylindrically hyperbolic are strong: these groups are SQ-universal \cite{DGO17}, have non-abelian free normal subgroups \cite{DGO17}, have infinite dimensional second bounded cohomology \cite{BestvinaFujiwara, Hamenstadt}, and have a well-developed small cancellation theory \cite{Hull}.

One of the most important outstanding questions about finitely generated acylindrically hyperbolic groups is whether the class is closed under quasi-isometry. One important difficulty in answering this questions is that there is currently little connection between the existence of generalized loxodromic elements and the geometry of the group. Given a finitely generated (non-virtually cyclic) group $G$, there exist both sufficient and necessary geometric conditions for an element $g$ (of infinite order) to be generalized loxodromic (see \S\ref{sec:intro} for precise definitions):
\begin{itemize}
 \item  If there is a finite symmetric generating set $S$ of $G$ and a constant $D$ such that $\fgen{g}$ is $D$--quasi-convex and $D$--strongly contracting in $\Cay(G,S)$, then $g$ is generalized loxodromic \cite{BBF15}.
 \item  If $g$ is generalized loxodromic, then for any finite symmetric generating set $S$ of $G$, $\fgen{g}$ is a Morse quasi-geodesic in $\Cay(G,S)$ \cite{Sisto_Morse}.
\end{itemize}
Note that by \cite[Theorem $4.19$]{ACGH2} the statement ``there exists a constant $D$ such that $\fgen{g}$ is $D$--strongly contracting in $\Cay(G,S)$'' can even depend on the choice of finite generating set $S$.

In \cite{ACGH1} $\rho$--contraction is introduced as a generalization of  strongly contracting
, and it is proved that a subset of a geodesic metric space is Morse if and only if it is $\rho$--contracting for some function $\rho$.
Moreover, in \cite[Theorem $4.15$]{ACGH2} it is demonstrated that (periodic) geodesics in finitely generated groups exhibit all possible types of $\rho$--contraction. 

\subsection{Saturation}
Our first task in this paper is to prove that contraction is (in general) the wrong measure to use to determine whether an element is generalized loxodromic.

\begin{theorem}\label{thm:contnotgenlox} Given any non-decreasing unbounded function $\rho:[0,\infty)\to[0,\infty)$, there is a group $G$ generated by a finite set $S$ and an element $a\in S$ such that the map $n\mapsto a^n$ is an isometric embedding of $\Z$ into $\Cay(G,S)$ and $\fgen{a}$ is $\rho'$--contracting for some $\rho'\preceq\rho$, but $a$ is not generalized loxodromic.
\end{theorem}
\begin{theorem}\label{thm:genloxnotcont} There is a group $G$ generated by a finite set $S$ and an element $a\in S$ such that the map $n\mapsto a^n$ is an isometric embedding of $\Z$ into $\Cay(G,S)$ and $a$ is generalized loxodromic, but $\fgen{a}$ is not $\rho$--contracting for any $\rho$ such that $\liminf_{r\to\infty}\rho(r)\log(r)/r=0$.
\end{theorem}
The first of these results was also proved in \cite[Theorem $6.4$]{ACGH2} using torsion as an obstruction. The examples we give can be made torsion-free, but rely heavily on the technology developed in \cite{ACGH2} to control the contraction. The $\log(r)$ term in Theorem \ref{thm:genloxnotcont} is intriguing. While it is clear that our method cannot be improved, we can show that for any function $\rho$ such that $\liminf_{r\to\infty}\rho(r)/r=0$ there is a countable group $G$ generated by a set $X$ and an element $a\in X$ such that the map $n\mapsto a^n$ is an isometric embedding of $\Z$ into $\Cay(G,X)$ and $a$ is generalized loxodromic, but $\fgen{a}$ is not $\rho$--contracting.

Let us introduce a new measure which is more closely linked to generalized loxodromic elements. Given two words $v,w\in F(S)$, with $v$ cyclically reduced, we define the \textbf{$w$--saturation of $v$} to be $\sat_w(v)=k/|v|_S$ where $k$ is the largest number of letters in a cyclic conjugate of $v$ which can be covered by (not necessarily disjoint) copies of cyclically reduced conjugates of the word $w$. For example $\sat_a(a^3ba^4)=\sat_{a^4}(a^3ba^4)=\frac78$ but $\sat_{a^8}(a^3ba^4)=0$. Our first result states that for any group, a strong form of $w$-saturation is sufficient to prevent $w$ from being a generalized loxodromic element.

\begin{theorem}\label{thm:saturatednotglox} Let $G$ be a group generated by a symmetric set $S$. If there exist cyclically reduced words $v_i\in F(S)$ such that $v_i=_G 1$ for each $i$, the image of the path $[1,v_i]$ in $\Cay(G,S)$ is quasi-isometric to a cycle of length $\abs{v_i}_S$ (with constants independent of $i$) and 
\begin{equation}\label{eq:sat}\lim_{n\to\infty}\limsup_{i\to\infty} \left(\sat_{w^n}(v_i)\right)> 0,
\end{equation} 
then $w$ is not generalized loxodromic.
\end{theorem}
When $(\ref{eq:sat})$ is satisfied we say $\set{v_i}$ is \textbf{$w$--saturated}. To prove Theorem \ref{thm:saturatednotglox} we use an obstruction given by \cite[Proposition 4.14]{DGO17}. From this we deduce the following corollary, first proved by Sisto \cite{Sisto_Morse}.

\begin{corollary}[\cite{Sisto_Morse}]\label{cor:notMorsenotsat} Let $G$ be a group generated by a symmetric subset $S$, and let $\fgen{w}$ be an undistorted infinite cyclic subgroup of $G$ with respect to the word metric on $S$. If $w$ is generalized loxodromic, then $\fgen{w}$ is Morse as a subset of $\Cay(G,S)$.
\end{corollary}

\subsection{Small cancellation groups}  The examples constructed in Theorems \ref{thm:contnotgenlox} and \ref{thm:genloxnotcont} are small cancellation groups (which are acylindrically hyperbolic by \cite{GruberSisto}).  Thanks to \cite{ACGH2}, we have a clear understanding of contraction in such groups: 
given a $C'(\frac16)$ small cancellation presentation $G=\fpres{S}{R}$ and an infinite order element $g\in G$, we have that $g$ being $\rho$-contracting is essentially equivalent to the statement that no relation in $R$ with length at most $n$ contains a subword equal to a power of a cyclically reduced conjugate of $g$ of length greater than $\rho(n)$.

For these reasons, small cancellation groups are a natural collection of groups in which to start studying connections between generalized loxodromic elements and geometric notions of negative curvature. Our next goal in this paper is to better understand when an element of a small cancellation group is (and is not) generalized loxodromic.

\begin{theorem}\label{thm:unsaturatedglox} Let $G=\fpres{S}{R}$ be a $C'(\frac16)$ small cancellation presentation, enumerate $R=\set{r_1,r_2,\ldots}$ and let $p_i$ be the length of the longest piece in $r_i$.  Let $w\in F(S)$ have infinite order in $G$. If there exists some $m$ such that 
\begin{equation}\label{eq:sparsesat}
 \limsup_{i\to\infty}p_i \cdot\sat_{w^m}(r_i)<\infty,
\end{equation} 
then $w$ is generalized loxodromic.
\end{theorem}
The strategy behind the proof is to show that $G$ is a subgroup of another $C'(\frac16)$ small cancellation group $G'$ such that $\fgen{w}$ is $D$--contracting in a Cayley graph of $G'$. 

We present two results explaining the sharpness of Theorem \ref{thm:unsaturatedglox}. The first is Theorem \ref{thm:saturatednotglox}: if $\limsup_{i\to\infty}\sat_{w^m}(r_i)$ has a positive lower bound which is independent of $m$, it is clear that ($\ref{eq:sparsesat}$) cannot hold, since $(p_i)$ is unbounded (for every $m$, some cyclically reduced conjugate of $w^m$ is a piece). The second says that it is always possible to build examples of groups which prevent the hypotheses of the above theorem from being strengthened.

\begin{theorem}\label{thm:exnotglox} For every unbounded function $\rho:\N\to\N$ there is a $C'(\frac16)$ small cancellation presentation $G=\fpres{a,b,c,d}{r_1,r_2,\ldots}$ such that the map $\Z\to\Cay(G,\set{a,b,c,d})$ given by $n\mapsto b^n$ is an isometric embedding and $p_i\cdot\sat_{b}(r_i)\preceq\rho(i)$ (where $p_i$ is the length of the longest piece in $r_i$), but $b$ is not generalized loxodromic.
\end{theorem}

Notice that Theorems \ref{thm:saturatednotglox} and \ref{thm:unsaturatedglox} greatly generalize the sufficient and necessary conditions for an element of a group to be a generalized loxodromic given earlier, but they certainly do not give a complete classification of generalized loxodromic elements.
\medskip

\subsection{Universally acylindrical groups}  Given an acylindrically hyperbolic group it is natural to ask not only which of its elements are generalized loxodromic, but also which of these generalized loxodromic elements can simultaneously act loxodromically with respect to single acylindrical action on a hyperbolic space. A particularly strong version of this is to ask whether a given group is \textbf{universally acylindrical}; that is, does it admit an acylindrical action on a hyperbolic space in which {\em every} generalized loxodromic element acts loxodromically.  Such an action is called a {\bf universal acylindrical action}. Non-elementary hyperbolic groups, mapping class groups, and, more generally, hierarchically hyperbolic groups are all universally acylindrical \cite{Bo08,ABD17}. In \cite{Abbott_Dunwoody}, the first author proved that Dunwoody's inaccessible group is not universally acylindrical, the first finitely generated example of such a group.

One can ask under what conditions  a small cancellation group is universally acylindrical.  The only existing construction of a hyperbolic space admitting a natural action of a small cancellation group is the coned-off graph described by Gruber and Sisto in \cite{GruberSisto}, which is defined to be the Cayley graph of the group with respect to the generating set consisting of the original generating set together with all subwords of relators. Coulon and Gruber show in  \cite{Coulon-Gruber} that the action of the group on this graph is acylindrical if and only if there is an upper bound on the proper powers appearing in relators and, moreover, that this action is a universal acylindrical action if and only if the same condition holds.  
%
%
%
%

Our goal is find more general hypotheses which determine when a small cancellation group is universally acylindrical. Our first result in this direction is a natural extension of Theorem \ref{thm:unsaturatedglox}. Given a cyclically reduced word $w\in F(S)$ and $n\in\N$ we define the \textbf{saturation by $n$th powers} of $w$ to be $\sat_n(w)=k/|w|_S$ where $k$ is the maximal number of letters in a cyclic conjugate of $w$ which can be covered by copies of $n$th powers of cyclically reduced elements of $F(S)$.  For example $\sat_1(a^3ba^4)=1$, $\sat_6(a^3ba^4)=\frac78$ and $\sat_8(a^3ba^4)=0$.

\begin{theorem}\label{thm:unifunsatimpliesUA} Let $G=\fpres{S}{R}$ be a $C'(\frac16)$ small cancellation presentation, and enumerate $R=\set{r_1,r_2,\dots}$. Let $p_i$ be the length of the longest piece in $r_i$. If there exists some $m$ such that $\lim_{i\to\infty}p_i\cdot \sat_m(r_i)<\infty$, then $G$ is universally acylindrical.
\end{theorem}

The second result uses techniques from Theorem \ref{thm:contnotgenlox} to construct uncountably many quasi-isometry classes of torsion-free groups which are acylindrically hyperbolic but not universally acylindrical.

\begin{theorem}\label{thm:tfnonUA} There exist $2^{\aleph_0}$ quasi-isometry classes of torsion-free $C'(\frac16)$ small cancellation groups which are not universally acylindrical.
\end{theorem}

\subsection{Questions}
\begin{enumerate}
\item Let us construct a small cancellation presentation which contains an element which does not satisfy the hypotheses of either of Theorems \ref{thm:saturatednotglox} or \ref{thm:unsaturatedglox}. Define $w_1=b\in F(a,b)$ and inductively define $w_i=w_{i-1}a^iw_{i-1}$.   Let $\set{v_i}\subset F(c,d)$ be a family of $C'(\frac{1}{11})$ small cancellation relations with $\abs{v_i}=11\abs{w_i}$. Set $R=\set{r_i=v_iw_i}$ and $G=\fpres{a,b,c,d}{R}$. A simple calculation shows that this presentation is $C'(\frac16)$ and $\lim_{i\to\infty}\sat_{a^n}(r_i)=\alpha_n>0$ for all $n$, but $\alpha_n\to 0$ as $n\to\infty$. Is $a\in G$ generalized loxodromic?

\item Is there an ``unsaturated criterion'' in the style of $(\ref{eq:sparsesat})$ such that given any finitely generated group $G$ generated by a finite set $S$ and an infinite order element $w\in G$ where every undistorted loop in $\Cay(G,S)$ is $w$-unsaturated, then $w$ is generalized loxodromic?

\item Is it true that for any function $\rho$ such that $\liminf_{r\to\infty}\rho(r)/r=0$ there is a group $G$ generated by a finite set $S$ and an element $a\in S$ such that the map $n\mapsto a^n$ is an isometric embedding of $\Z$ into $\Cay(G,S)$ and $a$ is generalized loxodromic, but $\fgen{a}$ is not $\rho$--contracting?
\end{enumerate}

\subsection*{Plan of the paper} After establishing some background results we prove all the sufficient statements in \S \ref{section:loxodromics}.   In particular, we  give a general construction which embeds one small cancellation group into another (Theorem \ref{supergroup}), from which we deduce Theorems \ref{thm:genloxnotcont}, \ref{thm:unsaturatedglox} and \ref{thm:unifunsatimpliesUA} as corollaries. The obstructions are contained in \S\ref{sec:notloxodromics}: we first prove Theorem \ref{thm:saturatednotglox}, then give a method of building non-examples from which we deduce Theorems \ref{thm:contnotgenlox}, \ref{thm:exnotglox} and \ref{thm:tfnonUA}.

\subsection*{Acknowledgements}
The authors are grateful to R\'{e}mi Coulon and Dominik Gruber for interesting conversations, and for sharing with us the results of their paper. The authors also thank the referee for several comments which improved the clarity of the paper.

%% file: preliminaries_revised.tex
Given a group $G$ which is generated by a symmetric set $X$ we denote the word metric on $G$ with respect to $X$ by $d_X$ and define $\abs{g}_X=d_X(1,g)$ for all $g\in G$.


\subsection{Small cancellation theory}
Let $S$ be a set, and let $R$ be a subset of a free group $F(S)\setminus\set{1}$ consisting of cyclically reduced elements which is closed under taking inverses and cyclic conjugates. Given $\lambda\in(0,1)$, we say $R$ satisfies the {\bf $C'(\lambda)$ small cancellation condition} if, given any distinct pair of elements $r,r'\in R$ which have reduced decompositions $r=us$ and $r'=us'$, we have $\abs{u}<\lambda\abs{r}$. Any word $u$ satisfying the above property is called a {\bf piece}.

It is useful to represent families of small cancellation relations graphically. To each element $r\in R$, we define a cyclic graph $C_r$ with $\abs{r}$ directed edges each of which is labeled by elements of $S$ in such a way that the word in $F(S)$ read clockwise from a fixed vertex $v$ is exactly $r$ (for an edge labeled $s$ directed clockwise we read $s$, if it is directed anticlockwise we read $s^{-1}$). Changing the cyclically reduced representative of $r$ changes the fixed vertex from which the label is read, while taking inverses corresponds to reading the label anticlockwise (or considering a reflection of the cycle). A piece is therefore a labeled, directed subpath which appears in at least two cycles.

We require two fundamental results from small cancellation theory.

\begin{lemma} Let $R\subset F(S)$ satisfy the $C'(\lambda)$ small cancellation condition for some $\lambda\leq\frac16$, and define $G=\fpres{S}{R}$
and $X=\Cay(G,S)$. For each $r\in R$, the natural label-preserving map $C_r\to X$ which sends $v$ to $1$ is an isometric embedding.
\end{lemma}
We call translates of the image of the cycle $C_r$ in $X$ {\bf relators} (with label $r$).

\begin{lemma}[Greendlinger's Lemma] Let $R\subset F(S)$ satisfy the $C'(\lambda)$ small cancellation condition, define $G=\fpres{S}{R}$ 
and let $w\in F(S)\setminus \{1\}$ satisfy $w=_G 1$. There is a cyclically reduced conjugate $w'$ of $w$ and some $r\in R$ with reduced decompositions $r=us$ and $w'=uv'$ such that $\abs{u}_S\geq (1-3\lambda)\abs{r}_S$.
\end{lemma}

\subsection{$\rho$--contraction}
A subset $Y$ of a geodesic metric space $X$ is \textbf{quasi-convex} if there exists a constant $M$ such that any geodesic between two points in $Y$ is contained in the $M$--neighborhood of $Y$. Moreover, $Y$ is \textbf{Morse} if, for every $K\geq 1$ and $C\geq 0$, there is a constant $N=N(K,C)$ such that any $(K,C)$--quasi-geodesic between two points in $Y$ is contained in the $N$--neighborhood of $Y$. When these conditions are satisfied we say that $Y$ is $M$--quasi-convex and $N$--Morse respectively.

Given a subset $Y$ of the vertex set of a graph $X$ we define the closest point projection $\pi_Y$ from $X$ to a subset of $Y$ by the rule
\[
 \pi_Y(x)=\setcon{y\in Y}{d(x,y)\leq d(x,Y)}.
\]
Given a non-decreasing function $\rho:[0,\infty)\to[0,\infty)$ such that $\rho(r)/r\to 0$ as $r\to \infty$, we say $Y$ is $\rho$\textbf{--contracting} if for any two points $x,z\in X$ with $d(x,z)\leq d(x,Y)$,
\[
	\diam\left(\pi_Y(x)\cup \pi_Y(z)\right) \leq \rho(d(x,Y)).
\]
We say $Y$ is 
$D$--\textbf{contracting} for a constant $D$ if $Y$ is $\rho$--contracting with $\rho(r) = D$ for all $r$. By \cite[Theorem 1.3]{ACGH1}, a quasi-convex subset $Y\subseteq VX$ is Morse if and only if it is $\rho$--contracting for some $\rho$. We compare contraction functions using the partial order $\preceq$ defined as follows: $\rho \preceq \rho'$ if there exists a constant $C$ such that
\[
 \rho(r)\preceq C\rho'(r)+C \quad \textup{for all }r.
\]
We write $\rho\asymp\rho'$ when $\rho\preceq\rho'$ and $\rho'\preceq\rho$.

We recall the key local-to-global theorem \cite[Theorem $4.1$]{ACGH2} which allows us to control contraction in small cancellation groups:

\begin{theorem}\label{ACGHcontracting} Let $R\subset F(S)$ satisfy the $C'(\frac{1}{6})$ small cancellation condition, set $G=\fpres{S}{R}$ 
and let $\alpha$ be a geodesic in $X=\Cay(G,S)$. There exists a sublinear function $\rho'$ such that $\alpha$ is $\rho'$--contracting in
$X$ if and only if there exists a sublinear function $\rho$ such that given any relator $C_r$ which intersects $\alpha$, we have $\abs{C_r\cap\alpha}\leq \rho(\abs{C_r})$.

Moreover, when either of the above conditions hold we may take $\rho'\asymp\rho$.
\end{theorem}


%% file: firstconedoff_revised2.tex
We begin by describing a hyperbolic graph on which $C'(\frac16)$--small cancellations groups act.  This space was introduced and studied in \cite{GruberSisto} in a more general context than we will present here.

Let $R\subset F(S)$ satisfy the $C'(\frac16)$ condition and let $X$ be the Cayley graph of $G=\fpres{S}{R}$ with respect to the generating set $S$. Let $L$ be the set of all initial subwords of words in $R$, that is, the set of all $u\in F(S)$ such that there is some $r\in R$ with reduced decomposition $r=uv$. The \textbf{coned-off graph} $\hat{X}$ is $\Cay(G,S\cup L)$, that is, the Cayley graph of $G$ with respect to the (typically infinite) generating set $S\cup L$.

By \cite[Proposition $3.2$]{GruberSisto}, $\hat X$ is hyperbolic.  Moreover, by \cite[Proposition $4.8$]{GruberSisto}, $G$ contains an element $g$ which acts loxodromically on $\hat X$ and satisfies the \emph{weak proper discontinuity (WPD) property} (see \cite{BestvinaFujiwara}). In \cite[Theorem 1.4]{Os16} Osin showed that the existence of an element $g$ satisfying the WPD property in an action of $G$ on a hyperbolic space is equivalent to $g$ being a generalized loxodromic element and $G$ being an acylindrically hyperbolic group.

However, in general the action of $G$ on $\hat X$ is not acylindrical \cite[Proposition $4.20$]{GruberSisto}. We recall the one situation in which acylindricity of this action is known. We say $R\subset F(S)$ is \textbf{power-free} if, for every $w\in F(S)\setminus\set{1}$, there exists an $n$ such that no $r\in R$ has reduced decomposition $r=w^nv$. If $n$ can be chosen independently of $w$, then we say $R$ is \textbf{uniformly power-free}. Coulon and Gruber prove that the action of every uniformly power-free (graphical) small cancellation group on its coned-off graph $\hat X$ is acylindrical \cite{Coulon-Gruber}.  They also show that every infinite order element of $G$ is loxodromic with respect to this action, and thus these groups are universally acylindrical. Our goal in this section is to give an alternative construction which enables us to detect  generalized loxodromic elements in small cancellation groups which are not necessarily uniformly power-free and to find more groups which are universally acylindrical.

\begin{definition}\label{def:sparsesat} Let $G=\fpres{S}{R}$ be a $C'(\frac16)$ small cancellation presentation. For $r\in R$, let $p_r$ be the length of the longest piece in $r$.  Let $\mathcal{W}\subseteq F(S)$, and let $\mathcal W^m=\{w^m\mid w\in \mathcal W\}$.  We define $l_{\mathcal W}(r)$ to be the maximum number of edges of $C_r$ which can be covered by (not necessarily disjoint) subpaths labeled by powers of cyclically reduced conjugates of elements of $\mathcal W$. The $\mathcal W$\textbf{--saturation} of $C_r$ is defined as
\[
\sat_{\mathcal W}(C_r)=\frac{l_{\mathcal W}(r)}{\abs{C_r}}\in[0,1].
\]
We say $R$ is $\mathcal W$\textbf{--sparse} if there exists $m\geq 1$ such that
\[
 \sup_{r\in R} \left(p_r\cdot \sat_{\mathcal W^m}(C_r)\right)<\infty,
\]
 and $R$ is $\mathcal W$\textbf{--saturated} if 
\[
 \lim_{m\to\infty}\limsup_{\abs{r}_S\to\infty} \left(\sat_{\mathcal W^m}(C_r)\right)>0.
\]
\end{definition}

Given a $C'(\frac16)$ small cancellation group $G=\fpres{S}{R}$ and a subset $\mathcal W \subseteq F(S)$ such that $R$ is $\mathcal W$--sparse, we first show how to construct a $C'(\frac16)$ small cancellation group $G'=\fpres{S'}{R'}$
 for which there exists an $l\geq 0$ such that no $r'\in R'$ has reduced decomposition $r'=w^lv'$ for any $w\in\mathcal W$, together with an embedding of $G$ as a subgroup of $G'$.
 

\begin{theorem} \label{supergroup}
Let $R\subset F(S)$ satisfy the $C'(\frac16)$ condition, and let $\mathcal W\subseteq F(S)$. If $R$ is $\mathcal W$--sparse, then there exists a set $S'$ containing $S$ and a $C'(\frac16)$ subset $R'\subseteq F(S')$ such that the following hold: the inclusion $S\into S'$ extends to a monomorphism $G=\fpres{S}{R} \to G'=\fpres{S'}{R'}$
, and there exists a constant $l\geq 0$ such that no $r'\in R'$ has reduced decomposition $r'=w^lv'$ for any $w\in\mathcal W\setminus\set{1}$.

Moreover, if $\abs{S}<\infty$, then 
every element of $\mathcal{W}$ which has infinite order in $G$ 
is a generalized loxodromic element of $G$.
\end{theorem}

\begin{proof}
Define a countably infinite set $Y=\set{y_1,y_2,\dots}$ such that $Y\cap S=\emptyset$.  We will see that the set $Y\cup S$ is our desired generating set $S'$.  Enumerate $R=\{r_1,r_2,\dots\}$, and let $C_i=C_{r_i}$.  
Our goal is to define a new set of disjoint cycles, whose labels give the desired set of relations $R'$.

Since $R$ is $\mathcal W$--sparse, there is a constant $m\geq 1$ such that 
\[
 \limsup_{i\to\infty} p_i\cdot\sat_{\mathcal W^m}(C_i)=D<\infty. 
\]
Notice that increasing $m$ does not increase $D$, so we may assume that $m>12D$. Fix $N$ such that for all $n\geq N$ we have $p_n\cdot\sat_{\mathcal W^m}(C_n)<2D$.

For each $n\geq N$ we construct a finite set of cycles $D_n^j$ as follows.  First, color every edge red in $C_n$ which is contained in a path whose label is a cyclically reduced conjugate of an element of $\mathcal W^m$.  There are at most $l_{\mathcal W^m}(r_n)$ such edges. Color the other edges blue.  In the procedure that follows, we will be adding new edges at each step which will be colored white; in the first step, there are no white edges.
 
For each $n\geq N$, define $C_n^0=C_n$. For each $j=1,2,\ldots$ in turn, 
if $C_n^{j-1}$ has no subpath of length at least $m$ consisting solely of red edges, then set $D_n^j=C_n^{j-1}$ and move on to $n+1$. Otherwise, notice that the maximal number of disjoint subpaths of length $m$ consisting solely of red edges is at most $l_{\mathcal W^m}(r_n)/m$.  Additionally, by our choice of constants we have $$\frac{m|C_n|}{l_{\mathcal W^m}(r_n)}=\frac{m}{\sat_{\mathcal W^m}(C_n)}>\frac{p_nm}{2D}>6p_n.$$  The pigeon-hole principle implies that there is a path $P_n^j$ in $C_n^{j-1}$  consisting of $m$ red edges followed by a mixture of blue, red and white edges satisfying the following properties: 
	\begin{itemize}
	\item there is no red subpath of length greater than $m$; and
	\item there are exactly $6p_n$ edges that are blue or red. 
		\end{itemize}
Recall that for $j=1$ there are no white edges.

\begin{figure}
\begin{center}
\resizebox{2in}{!}{
  \centering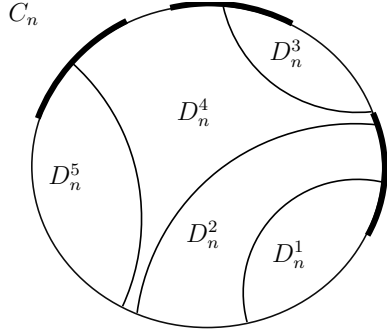} \\
  \end{center}
	\caption{The darkened portions of the cycle $C_n$ are composed of red edges, the interior arcs are white, and all other edges are blue.} \label{newcycles}
\end{figure}

	 Add a path $\alpha_n^j$ of length $p_n$ connecting the last vertex of $P_n^j$ to the first vertex of $P_n^j$, and label its edges using the first $p_n$ unused elements of $Y$.  Color these new edges white. The graph obtained by adding the path $\alpha_n^j$ has two simple cycles containing $\alpha_n^j$.   Denote the simple cycle containing $P_n^j$ by $D_n^j$, and denote the other by $C_n^j$. There are only finitely many red edges in $C_n$ so this process will terminate for each $n$.  (See Figure \ref{newcycles}.)

Let $S'=S\cup Y$, and let $R'$ be the union of $\{r_i \mid i<N\}$ and the set of labels of the cycles $D_n^j$.  It is clear that $S\into S'$.  We claim that this inclusion induces a monomorphism $\phi\colon G=\fpres{S}{R}\to G'=\fpres{S'}{R'}$.  To see that this inclusion induces a homomorphism $\phi\colon G\to G'$, we show that all relations in $G$ hold in $G'$.  If $i<N$, then $r_i\in R'$.  If $i\geq N$, then $r_i$ is the label of the cycle $C_i$, and by construction, $r_i$ is the label of the boundary of a product of finitely many cycles $D_i^j$, whose labels are elements of $R'$.   In either case, the relation $r_i$ holds in $G'$.  To further prove that $\phi$ is a monomorphism, suppose $v\in F(S)$ is trivial in $G'$.  Then $v$ is the label of a product of cycles $C_i$ and $D_i^j$.  However, by construction, the only way the label of such a product can be an element of $F(S)$ is if $v$ can also be written as the label of a product of cycles $C_n$. Therefore, $v=_G1$.

We now show that $R'$ is a $C'(\frac16)$ subset of $F(S')$ and that there exists some $l\geq 0$ such that no $r\in R'$ has a reduced decomposition $r'=w^lv'$ for any $w\in\mathcal W$.

By construction, a piece in $R'$ is either a piece in $R$ or the label of a subpath of some $\alpha_n^j$. If it is a piece in $R$ then it is the label of a subpath of some $C_n$.  If $n<N$ there is nothing to prove. Otherwise, the piece has length at most $p_n$, while the length of any $D_n^j$ is at least $7p_n$ by construction.  Thus $R'$ is a $C'(\frac16)$ subset of $F(S')$.

Fix $l>m$ so that no subpath of any $C_n$ with $n\leq N$ has label contained in $\mathcal W^l\setminus\set{1}$, and suppose there exists some $r'\in R'$ which has a reduced decomposition $r'=w^lv'$ such that  $w\in\mathcal W$.  Then $r'$ must be the label of a subpath of some $D_n^j$ which contains at least $l$ consecutive red edges.  However, by construction, any subpath of any $D_n^j$ consisting solely of red edges has length at most $m$, so there are no such subpaths.

Finally, suppose $\abs{S}<\infty$.  Since there are only finitely many words of a given length,  $p_n\to\infty$ as $n\to\infty$.  Thus there is a last cycle which contains a path of any fixed length which has not appeared as a subpath of any previous cycle. For each $k$ let $l_k$ be the length of the longest cycle $C_n$ with $p_n\leq k$.  It follows that for all $g\in G$ such that $d_{S'}(1,g)\leq k$, we have $d_{S}(1,g)\leq l_kk$, so $G\cap B^{S'}(1; k)$ has finitely many elements.  Therefore, the image of $G$ in $G'$ is metrically proper, so every element of infinite order in $G$ acts weakly properly discontinuously on $\Cay(G',S')$. Moreover, every infinite order element of $\mathcal W$ has strongly contracting orbits in $\Cay(G',S')$ by \cite[Corollary $4.14$]{ACGH2}.  Therefore, by applying \cite[Theorem H]{BBF15} and \cite[Theorem 1.4]{Os16}, we conclude that all such elements are generalized loxodromic elements of $G$.
\end{proof}

Theorems $\ref{thm:unsaturatedglox}$ and $\ref{thm:unifunsatimpliesUA}$ are special cases of the above theorem, where  $\mathcal W$ is a single infinite order element or the whole of $G=\fpres{S}{R}$, respectively. In the second case it is easy to see that $R'$ is uniformly power-free.  Thus \cite[Theorem 5.10]{Coulon-Gruber} implies $G'$ (and hence $G$) acts acylindrically on its coned-off graph $\widehat{X'}$  and every infinite order element of $G'$ acts loxodromically.  Hence the action $G\curvearrowright \widehat{X'}$ is universally acylindrical.

\begin{remark} To conclude that elements of $\mathcal W$ are generalized loxodromic in the case where $\abs{S}=\infty$, one can begin with the argument above and attempt to prove directly that each infinite order element of $\mathcal W$ acts as a WPD isometry of $G'$. The necessary ingredients would likely include an understanding of quadrangles with two sides labeled by powers of a word. Since these are not geodesics in general (though they are close, see \cite[Theorem $4.2$]{Gruber_C6SQ}) there is an extra complication to be dealt with before the classification of quadrangles in small cancellation groups can be used. Since most interest in these groups is focused on finitely generated examples, we will not pursue this further here.
\end{remark}

We are now ready to prove Theorem \ref{thm:genloxnotcont}. 
\begin{proof}[Proof of Theorem \ref{thm:genloxnotcont}]
Set $S=\set{a,b,c,d}$, and for each $k\geq 7$ fix $2^{k+4}$ distinct words $w_k^i(b,c)$ in $F(b,c)$ of length exactly $k$. Now define
\[
 r_k = a^{2^k} \Pi_{i=1}^{2^{k+4}} dw_k^i.
\]
Set $G=\fpres{S}{r_7,r_8,\ldots}$. It is easy to check that this presentation satisfies $C'(\frac16)$, since $\abs{r_k}= (16k+17)2^k$ and the longest possible piece in this relator is $a^{2^k}dw_k^1$ 
which has length $2^k+k+1\leq 2\cdot2^{k}$. Since relators in $C'(\frac16)$ groups are isometrically embedded in the Cayley graph, the map $\Z\to\Cay(G,S)$ given by $n\mapsto a^n$ is an isometric embedding. Using \cite[Theorem $4.1$]{ACGH2}, we see that $\fgen{a}$ is not $\rho$--contracting in $\Cay(G,S)$ unless $\rho(r)\succeq r/\log(r)$. It remains to prove that $a$ is a generalized loxodromic element.

Next let $X=\set{x_1,x_2,\ldots}$ be a countably infinite set, define $S'=S\sqcup X$ and for each $k\geq 5$ and $0\leq j\leq 2^k-1$ define the following elements of $F(S')$:
\[
 R^j_k = ax^j_k \left[\Pi_{i=16j+1}^{16(j+1)} dw_k^i(b,c)\right] (x^{j+1}_k)^{-1}
\]
where each $x^j_k$ is the product of the first $k$ elements of $X$ which have not already appeared in some $x^i_l$ with either $i<j$ and $l\leq k$ or $l<k$, except for $x^0_k$ and $x^{2^k}_k$, which are equal to the identity.

We first claim that $G'=\fpres{X}{\mathcal R}$ is $C'(\frac18)$ where $\mathcal R$ is the set of cyclically reduced conjugates of the $R^j_k$ and their inverses. Since the $x^j_k$ are words in disjoint alphabets, any piece which is an initial subword of some $R^j_k$ must either be some $(x^j_k)^{\pm 1}$ or a subword of a word of the form $(w_k^i d w_k^{i+1})^{\pm 1}$ or $(w_k^{16(j+1)} a w_k^{16j+1})^{\pm 1}$. Hence any piece $u$ contained in a relation $R^j_k$ satisfies $\abs{u}\leq 2k+1$ and $\abs{R^j_k}\geq 16(k+1)$. Since $2k+1 <\frac18\cdot 16(k+1)$ the claim is verified.

The proof of Theorem \ref{supergroup} can be repeated verbatim to deduce that the inclusion $S\into S'$ extends to a monomorphism $G\to G'$, that $\fgen{a}$ is strongly contracting in $X'=\Cay(G',S')$ and that the action of $a$ on $X'$ is weakly properly discontinuous. It follows that $a$ is a generalized loxodromic element of $G$ by again applying \cite[Theorem H]{BBF15} and \cite[Theorem 1.4]{Os16}.
\end{proof}

It is clear that the $\log(r)$ term in statement of Theorem \ref{thm:genloxnotcont} is necessary for our method of proof, since any word in a finitely generated free group of length $n$ must contain two identical subwords of length $\asymp \log(n)$ by the pigeon-hole principle.  However, if we do not require that our group be finitely generated, then for any function $\rho$ such that $\liminf_{r\to\infty}\rho(r)/r=0$, we can find a countable group $G$ generated by a set $Y$ and an element $a$ which is a generalized loxodromic element of $G$ but is not $\rho$--contracting in $\Cay(G,X)$.  We end this section by describing this construction.
\begin{example}
Fix a function $\rho(r)$ such that $1\leq \rho(r)\leq r$ for all $r$ and $\liminf_{r\to\infty}\rho(r)/r=0$. Define an unbounded function $\tau:\N\to\N$ such that $\tau(r)^2 \leq r/\rho(r)$. Now let $Y=\set{y_1,y_2,\ldots}$ be a countably infinite set. Fix a sequence $(n_k)_{k\in\N}\subset\N$ such that for each $n_k$, $l_k:=\lfloor n_k/\tau(n_k)\rfloor\geq 6$ and $\tau(n_k)\geq k$, and define the following elements of $F(\set{a}\sqcup Y)$:
\[
 r_{k} = a^{l_k} w_k,
\]
where $w_k$ is the product of the first $2n_k$ elements of $Y$ which have not previously appeared in any $w_l$ with $l<k$. Now
\[
 \abs{r_k} = l_k + 2n_k > 2kl_k,
\]
so $G=\fpres{\set{a}\sqcup Y}{r_3,r_4,\ldots}$ is $C'(\frac16)$ since the only pieces are powers of $a$. Again, the map $n\to a^n$ defines an isometric embedding of $\Z$ into $\Cay(G,\set{a}\sqcup Y)$. Suppose for a contradiction that $\fgen{a}$ is $\rho$--contracting, so by Theorem \ref{ACGHcontracting} it is locally $\rho'$--contracting for some $\rho'\asymp \rho$. This means that for any (cyclically reduced conjugate of) a relation $r_k$ with reduced decomposition $a^{l_k}w_k$ we must have $\rho'(\abs{w_k}/2)\geq l_k$\footnote{Specifically, we have $\rho'(\lceil \abs{w_k}/4\rceil)\geq l_k$. Since the function $\rho'$ is non-decreasing, the condition $\rho'(\abs{w_k}/2)\geq l_k$ is weaker.}. But the relations $r_k$ above satisfy
\[
 l_k\leq \rho'(\abs{w_k}/2) = \rho'(n_k) \asymp \rho(n_k) \preceq \frac{n_k}{\tau(n_k)^2} \preceq \frac{n_k}{k\tau(n_k)} \preceq \frac{l_k}{k},
\]
which is a contradiction for sufficiently large $k$.

Arguing as before we can construct a supergroup $G'$ of $G$ such that $\fgen{a}$ is strongly contracting in a Cayley graph $X'$ of $G'$ and $G$ acts metrically properly on $X'$ (the $x^j_k$ should be chosen to have length $\tau(n_k)$). Hence, $a$ is a generalized loxodromic element of $G$.
\end{example}

%% file: embedding1.pdf_tex
\begingroup%
  \makeatletter%
  \providecommand\color[2][]{%
    \errmessage{(Inkscape) Color is used for the text in Inkscape, but the package 'color.sty' is not loaded}%
    \renewcommand\color[2][]{}%
  }%
  \providecommand\transparent[1]{%
    \errmessage{(Inkscape) Transparency is used (non-zero) for the text in Inkscape, but the package 'transparent.sty' is not loaded}%
    \renewcommand\transparent[1]{}%
  }%
  \providecommand\rotatebox[2]{#2}%
  \ifx\svgwidth\undefined%
    \setlength{\unitlength}{172.37767257bp}%
    \ifx\svgscale\undefined%
      \relax%
    \else%
      \setlength{\unitlength}{\unitlength * \real{\svgscale}}%
    \fi%
  \else%
    \setlength{\unitlength}{\svgwidth}%
  \fi%
  \global\let\svgwidth\undefined%
  \global\let\svgscale\undefined%
  \makeatother%
  \begin{picture}(1,0.87518722)%
    \put(0,0){\includegraphics[width=\unitlength,page=1]{embedding1.pdf}}%
    \put(-0.00493939,0.82030968){\color[rgb]{0,0,0}\makebox(0,0)[lb]{\smash{$C_n$}}}%
    \put(0.69331173,0.17819936){\color[rgb]{0,0,0}\makebox(0,0)[lb]{\smash{$D_n^1$}}}%
    \put(0,0){\includegraphics[width=\unitlength,page=2]{embedding1.pdf}}%
    \put(0.46523966,0.23083139){\color[rgb]{0,0,0}\makebox(0,0)[lb]{\smash{$D^2_n$}}}%
    \put(0,0){\includegraphics[width=\unitlength,page=3]{embedding1.pdf}}%
    \put(0.68629408,0.71504567){\color[rgb]{0,0,0}\makebox(0,0)[lb]{\smash{$D^3_n$}}}%
    \put(0,0){\includegraphics[width=\unitlength,page=4]{embedding1.pdf}}%
    \put(0.44067821,0.55714963){\color[rgb]{0,0,0}\makebox(0,0)[lb]{\smash{$D^4_n$}}}%
    \put(0.10032468,0.38521848){\color[rgb]{0,0,0}\makebox(0,0)[lb]{\smash{$D^5_n$}}}%
  \end{picture}%
\endgroup%

%% file: nonGL_revised.tex
In this section, we give a sufficient condition for an element of a group to fail to be generalized loxodromic, proving Theorem \ref{thm:saturatednotglox} and Corollary \ref{cor:notMorsenotsat}.  Additionally, we construct a particularly nice family of examples which are used to prove Theorems \ref{thm:contnotgenlox}, \ref{thm:exnotglox}, and \ref{thm:tfnonUA}.

\subsection{Hyperbolically embedded subgroups}

We briefly review the definition of a hyperbolically embedded subgroup and refer the reader to \cite{Os16} and \cite{DGO17} for a more complete discussion.

Let $H$ be a subgroup of $G$ and $X\subseteq G$ a relative generating set, i.e., a subset such that $G=\langle X\cup H\rangle$.  We consider the Cayley graphs $\Cay(G,X\sqcup H)$ and $\Cay(H,H)$, and we naturally think of the latter as a subgraph of the former.  We define a metric $\widehat d\colon H\times H\to [0,\infty]$ as follows.  Given $h_1,h_2\in H$, let $\widehat d(h_1,h_2)$ be the length of the shortest path from $h_1$ to $h_2$ in $\Cay(G,X\sqcup H)$ that does not include any edges from $\Cay(H,H)$.  If no such path exists, then let $\widehat d(h_1,h_2)=\infty$.  

We say $H$ is {\bf hyperbolically embedded in $G$ with respect to $X$}, and write $H\hookrightarrow_h (G,X)$, if the following hold: 

\begin{enumerate}
\item $\Cay(G,X\sqcup H)$ is hyperbolic; and 

\item $(H,\widehat{d})$ is a locally finite metric space, i.e., any ball of finite radius in $H$ with respect to the metric $\widehat d$ contains finitely many elements. 
\end{enumerate}

We first gather some facts about hyperbolically embedded subgroups.  Given a generalized loxodromic element $w\in G$, let $E(w)$ be the maximal virtually cyclic subgroup containing $w$.  The following lemma is \cite[Cor. 2.9]{Hull}.

\begin{lemma} \label{Ew}
Suppose $X$ is a subset of $G$ such that $\Cay(G,X)$ is hyperbolic, $G\curvearrowright \Cay(G,X)$ is acylindrical, and $w\in G$ is loxodromic with respect to this action.  Then $E(w)\hookrightarrow_h (G,X)$.
\end{lemma}

The next lemma follows from \cite[Lemma 4.11(b)]{DGO17} and \cite[Theorem 4.24(a)]{DGO17}.

\begin{lemma}  \label{finiteset}
If $H$ is a subgroup of $G$ and $X$ a subset of $G$ such that $H\hookrightarrow_h (G,X)$, then there is a finite subset $Y\subset H$ such that $\widehat d$ is biLipschitz equivalent to a word metric $d_Y$.
\end{lemma}

We require one additional lemma, which is a simplified version of \cite[Proposition $4.14$]{DGO17}. If $H\hookrightarrow_h(G,X)$, then a subpath $p$ of a path $q$ in $\Cay(G,X\sqcup H)$ is an \emph{$H$--component} if the label of $p$ is a word in the alphabet $H$ and $p$ is not contained in a longer subpath of $q$ whose label is a word in the alphabet $H$.  Two $H$--components $p_1,p_2$ of a path $q$ are \emph{connected} if there is an edge $c$ labeled by an element of $H$ whose initial vertex is  on $p_1$ and whose terminal vertex is on $p_2$.  Algebraically, this corresponds to every vertex of $p_1$ and $p_2$ belonging to the same coset of $H$.  An $H$--component $p$ of a path $q$ is \emph{isolated} if it is not connected to any other $H$--component of $q$.

For each $\mu\geq 1$, $c\geq 0$, and $n\geq 2$, let $\mathcal Q_{\mu,c}(n)$ denote the set of all pairs $(\mathcal P,I)$ where $\mathcal P=p_1\dots p_n$ is an $n$--gon in $\Cay(G,X\sqcup H)$ and $I$ is the distinguished subset of sides $\{p_1,p_2,\dots, p_n\}$ such that each $p_i\in I$ is a $(\mu,c)$--quasi-geodesic and an isolated $H$--component in $\mathcal P$.  Note that we allow sides of $\mathcal P$ to be trivial.  Given $(\mathcal P,I)\in \mathcal Q_{\mu,c}(n)$, let \[s(\mathcal P,I)=\sum_{p_i\in I}\widehat d((p_i)_-,(p_i)_+),\] and let \[s_{\mu,c}(n)=\sup_{(\mathcal P,I)\in \mathcal Q_{\mu,c}(n)}s(\mathcal P,I).\]

\begin{lemma}\label{lem:DGO4.14} Let $H$ be a subgroup of $G$ and $X$ a subset of $G$ such that $\Cay(G,X\sqcup H)$ is hyperbolic.  Then for any $\mu\geq 1$ and $c\geq 0$, there exists a constant $D=D(\mu,c)>0$ such that $s_{\mu,c}(n)\leq Dn$ for any $n\in \mathbb N$. 
\end{lemma}

Notice that the hypotheses of Lemma \ref{lem:DGO4.14} are satisfied if $H\hookrightarrow_h(G,X)$.

\subsection{Detecting non-generalized loxodromic elements}\label{sec:notloxodromics}

In this section, we give a sufficient condition under which an element is not generalized loxodromic. Recall that given a family of cyclic graphs $\set{C_n}$ with directed edges labelled by elements of $S$, and an element $w\in F(S)$ we say $\set{C_n}$ is $w$--saturated if
\[
 \lim_{m\to\infty}\limsup_{\abs{C_n}_S\to\infty} \left(\sat_{w^m}(C_n)\right)>0,
\]
where $\sat_{v}(C_n)= l_v(n)/\abs{C_n}$ and $l_v(n)$ is the maximal number of edges in $C_n$ which lie on (not necessarily disjoint) subpaths of $C_n$ whose labels are powers of cyclically reduced conjugates of $v$ (see Definition \ref{def:sparsesat}).
%

\begin{theorem}\label{thm:notgenlox} Let $G$ be a group generated by a 
symmetric set $S$, and suppose there are uniformly $(\mu',c')$-quasi-isometric simplicial maps $\psi_n:C_n\to \Cay(G,S)$.  
Denote the (not necessarily embedded) cycle $\psi(C_n)$ in $\Cay(G,S)$ by $D_n$, and denote the label of $D_n$ (as a word in $F(S)$) by $r_n$. If $\Gamma=\bigsqcup_n D_n$ is $\{w\}$--saturated for some $w\in F(S)$, then $w$ is not a generalized loxodromic element of $G$.
\end{theorem}
\begin{proof}
Without loss of generality, we may assume that $w$ is cyclically reduced. Suppose for a contradiction that $w$ is a generalized loxodromic element and choose $X\subset G$ such that $\Cay(G,X)$ is hyperbolic, $G\curvearrowright \Cay(G,X)$ is acylindrical and $w$ is loxodromic with respect to this action. Let $E(w)$ be the maximal virtually cyclic subgroup containing $w$. Choose a constant $M$ so that $E(w)$ is contained in the $M$--neighbourhood of $\fgen{w}$ in $\Cay(G,X)$. By Lemma \ref{Ew}, $E(w)\hookrightarrow_h(G,X)$.  Let $Y\subset E(w)$ be the finite subset provided by Lemma \ref{finiteset}. It follows immediately that there is a constant $\kappa>0$ such that $d_Y(a,b)\geq \kappa d_S(a,b)$ holds for all $a,b\in E(w)$.


For each sufficiently large $l$, we choose some $N=N(l)$ such that $\sat_{w^l}(D_N)\geq\varepsilon>0$. We will now replace $D_N$ by a $2k(l)$--gon, where the images of odd numbered sides are isolated $E(w)$--components, while even numbered sides are geodesics in $\Cay(G,X\sqcup E(w))$ connecting end points of consecutive odd numbered sides. We will then show that these polygons cannot satisfy Lemma \ref{lem:DGO4.14}, contradicting the assumption that $w$ is a generalized loxodromic element.

To do this, consider a maximal collection $\mathcal P_N$ of disjoint subpaths  of $C_N$ such that for each $P\in\mathcal P_N$, $\psi_N(P)$ is labeled by $w^k$ for some $k\geq l$. Extend each $\psi_N(P)$ to a maximal length subpath $L$ of $D_N$ which is contained in the same coset of $E(w)$ as the end vertices of $\psi_N(P)$. Notice that for all $N$ large enough, $D_N$ cannot be contained in a single $E(w)$ coset and that  multiple $P$'s may define the same $L$.  Let $L_1,\ldots L_k$ be the collection of all such paths $L$, and for each $L_i$ let $a_i,b_i$ be the end vertices of the path.

The cycles $D_N$ are uniformly quasi-isometrically embedded in $\Cay(G,S)$, so
\[
 \sum_{i=1}^k l_S(L_i) \asymp \sum_{i=1}^k d_S(a_i,b_i) \preceq \sum_{i=1}^k d_Y(1,a_i^{-1}b_i) \asymp \sum_{i=1}^k \widehat d(a_i,b_i)
\]
The Cayley graph $\Cay(G,X)$ is a subgraph of $\Cay(G,X\sqcup E(w))$, so we can consider the $L_i$ as paths in $\Cay(G,X\sqcup E(w))$.  In $\Cay(G,X\sqcup E(w))$ each $L_i\in \mathcal L_N$ can be replaced by an edge $e_i$, and then the $e_i$ can be connected by geodesics $\gamma_i$ to form a $2k(l)$--gon $\mathcal P=(e_1,\gamma_1,\ldots, e_{k(l)},\gamma_{k(l)})$, where the $e_i$ are all isolated $E(w)$--components. Setting $I=\setcon{e_i}{1\leq i\leq k(l)}$ we see that
\[
s(\mathcal P,I) =\sum_{i=1}^k \widehat d(a_i,b_i) \succeq \sum_{i=1}^k l_S(L_i) \succeq \sum_{P\in\mathcal P_N} \abs{P} \succeq l_N.
\]
The last step uses the fact that the $D_N$ are $w$--saturated. Thus $k(l)/l_N$ is bounded away from 0 by Lemma \ref{lem:DGO4.14}.

On the other hand, each isolated component has length which grows at least linearly with $l$, and so there are $k(l)\preceq l_N/l$ such components.  Therefore, $k(l)/l_N\to 0$ as $l\to\infty$,  which is a contradiction.
\end{proof}

From this, we recover a theorem of Sisto \cite[Theorem 1.1]{Sisto_Morse}.

\begin{corollary}\label{cor:genloxMorse} Let $G$ be a group generated by a finite set $S$, and suppose there is some $w\in G$ of infinite order such that the orbit of $w$ is not Morse in $X=\Cay(G,S)$. Then $w$ is not a generalized loxodromic element of $G$.
\end{corollary}
\begin{proof} Suppose towards a contradiction that $w$ is a generalized loxodromic element of $G$.   Without loss of generality we may assume $w$ is cyclically reduced as an element of $F(S)$. Since $\fgen{w}$ is not Morse, there exist constants $K\geq 1$ and $C\geq 0$ and continuous $(K,C)$-quasi-geodesics $\alpha_n$ from $1$ to $w^{m(n)}$ which are not contained in the $n$-neighbourhood of $\fgen{w}$.

Fix any geodesic $[1,w]$ in $X$ and consider the bi-infinite path $P$ obtained by concatenating the translates of $[1,w]$ by $\fgen{w}$. Let $k=l_S(w)$ and assume $n>3k$. Increasing $K$ and $C$ if necessary, we ensure that $P$ is a continuous $(K,C)$-quasi-geodesic. Let $x\in\alpha_n$ be a vertex at maximal possible distance from the restriction of $P$ to a path $P_n$ from $1$ to $w^{m(n)}$. It follows that $l=d_S(x,P_n)\geq n-k$. Following $\alpha_n$ in each direction from $x$, let $y$ and $y'$ be the first vertices satisfying \begin{align}\label{eqn:yPn} 2Kd_S(\cdot,P_n)\leq d_S(\cdot,x).\end{align} Choose any geodesics $\gamma$ and $\gamma'$ from $y$ and $y'$, respectively, to any closest points $z$ and $z'$ on $P_n$. Denote the quasi-geodesic quadrangle with sides $\gamma,\gamma'$ and the restrictions of $P_n$ and $\alpha_n$ to paths from $z$ to $z'$ and $y$ to $y'$, respectively, by $D_n$. (See Figure \ref{fig:qgeo1}).

\begin{figure}
\resizebox{3.5in}{!}{
  \centering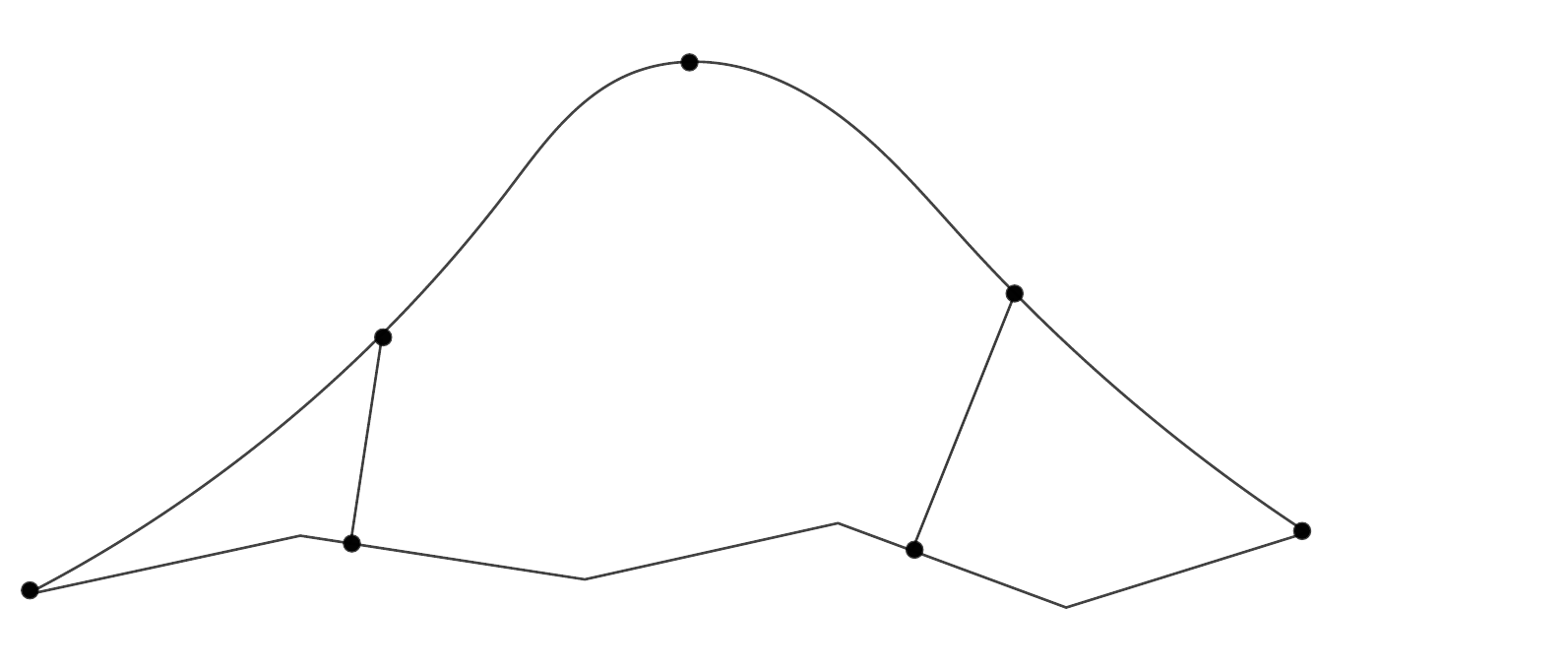} \\
	\caption{} 
	\label{fig:qgeo1}
\end{figure}

We will now argue, using the proof of Theorem \ref{thm:notgenlox} as our model, that the cycles $D_n$ provide an obstruction to $w$ being a generalized loxodromic element.

Firstly, we argue that $\set{D_n}$ is $w$--saturated.  For this it suffices to show that $\abs{P\cap D_n}=d_S(z,z')$ is bounded from below by some positive multiple of $\abs{D_n}$. Notice that since $y,y'$ lie on the continuous $(K,C)$-quasi-geodesic $\alpha_n$ which we may assume is parametrized by arc length, we have
\begin{align} \label{eqn:yy'}
 d_S(y,y') \geq \frac{1}{K} \left(d_S(y,x)+d_S(x,y')\right) - C.
\end{align}
Therefore, by the triangle inequality, (\ref{eqn:yPn}), and (\ref{eqn:yy'}), we have
\[
 \begin{array}{rl} d_S(z,z') 
 \geq & d_S(y,y') - d_S(y,P_n) - d_S(y',P_n)  \medskip \\
 \geq & \frac{1}{2K}\left(d_S(y,x)+d_S(x,y')\right) - C,
 \end{array}
\]
so it suffices to show that $d_S(x,y)$ is uniformly proportional to $\abs{D_n}$. By (\ref{eqn:yPn}) and the fact that $d_S(x,y)+d_S(y,P_n)\geq l$, we have
\[
 d_S(x,y) \geq 2Kd_S(y,P_n) \geq 2Kl - 2Kd_S(x,y)
\]
and thus $d_S(x,y) \geq \frac{2Kl}{2K+1}$. 

Next we bound $\abs{D_n}$. Since $y\in\alpha_n$, $d_S(y,P_n)\leq l$, so $d_S(x,y)\leq 2Kl$. Therefore the subpath of $\alpha_n$ from $y$ to $y'$ has length at most $4K^2L + C$. Moreover, $d_S(z,z')\leq 4Kl+2l$, so the length of the subpath of $P$ connecting $z$ to $z'$ has length at most $4K^2l + 2Kl + C$. Thus
\[
\abs{D_n} \leq 8K^2l+2Kl+2l+2C,
\]
completing the argument (for large enough $n$).

Now it need not be the case that the $D_n$ are uniformly quasi-isometrically embedded in $\Cay(G,S)$, but even so the technique employed in the proof of Theorem \ref{thm:notgenlox} can be made to work. As in that proof, consider a maximal subpath $L$ of the cycle $D_n$ containing $D_n\cap P_n$ whose end-vertices are contained in $E(w)$. By construction (for sufficiently large $n$) this path is contained in $D_n\cap P_n$ plus two subpaths of $[y,z]$ and $[y',z']$ of uniformly bounded length. Let $\mathcal P$ be the quadrangle consisting of sides $L$, the subpaths of $[y,z]$ and $[y',z']$ not already included in $L$ and the restriction of $\alpha_n$ to a path from $y$ to $y'$. Let $a,b$ be the end vertices of $L$. Arguing as before we see that
\[
s(\mathcal P,I) = \widehat d(a,b) \succeq d_S(a,b) \succeq \abs{D_n\cap P_n} \succeq \abs{D_n}.
\]
This contradicts Lemma \ref{lem:DGO4.14}.
\end{proof}

We immediately deduce Corollary \ref{cor:notMorsenotsat} from this.

Let us finish this section by giving a simple example of a group with an infinite order element which is Morse but cannot be generalized loxodromic by Theorem \ref{thm:notgenlox}. 

\begin{example}  Let $R$ be the set of all cyclically reduced conjugates of the following words (indexed by natural numbers $n\geq 12$), and their inverses, in $F(a,b)$:
\[
 r_n = ab^{n^2+1}ab^{n^2+2}\ldots ab^{(n+1)^2}
\]
and define $G=\fpres{S}{R}$. Firstly, $R$ satisfies the $C'(\frac16)$ condition, since any piece in $r_n$ is a subword of $b^{(n+1)^2-1}ab^{(n+1)^2}$ or its inverse and so has length at most $2(n+1)^2$, while $\abs{r_n}\geq (n+1)^3$.  Notice that $2(n+1)^2<\frac16 (n+1)^3$ because we assumed $n\geq 12$. The corresponding collection of cycles $C_n$ embed isometrically into $X=\Cay(G,S)$, and $\Gamma=\bigsqcup_n C_n$ is $\set{b}$--saturated, so $b$ is not generalized loxodromic by Theorem \ref{thm:notgenlox}. However, the map $\Z\to X$ defined by $n\mapsto b^n$ is an isometric embedding and the geodesic with vertex set $\fgen{b}$ is $\rho$--contracting for $\rho(r)\asymp r^{\frac23}$ (see Theorem \ref{ACGHcontracting}).
\end{example}

\subsection{A construction of non-examples}
Here we will give a versatile construction which allows us to prove that the seemingly arbitrary hypotheses of Theorem \ref{thm:unsaturatedglox} are sharp. This construction can be used to build groups with elements that are contracting but not generalized loxodromic, as well as torsion-free groups which  are not universally acylindrical. 

\begin{example} \label{ex:construction} In each case we begin with a collection of distinct words $\setcon{g_n}{n\in\N}\subseteq F(a,b)\setminus\set{1}$ and a function $f:\N\to\N$ such that $f(n)\geq 12$ for all $n$.  From this data we construct elements of $F(a,b,c,d)$ in the following way:
\begin{equation}\label{eqn:defineRn}
R_n = \prod_{i=1}^{f(n)} [c^{d^{i}}g_n(c^{-1})^{d^{i}},c^{d^{(f(n)+i)}}g_n(c^{-1})^{d^{(f(n)+i)}}]
\end{equation}
where we use the conjugation notation $x^y=yxy^{-1}$.

Note that to obtain a cyclically reduced word conjugate to $R_n$ we need only cancel words of the form $dd^{-1}$ and $d^{-1}d$. Let $C_n$ be the labeled cyclic graph corresponding to $R_n$. Set $S=\set{a,b,c,d}$ and define $R$ to be the set of cyclically reduced conjugates of the $R_n$ and their inverses. Define $G=F(a,b,c,d)/\fgen{R}$.  
  Our first goal is to prove that $R$ satisfies $C'(\frac16)$.

For each $1\leq k\leq 2f(n)$, and each $\eps\in\set{\pm 1}$ there is exactly one subpath of $C_n$ with each of the following labels:
\begin{equation}\label{eqn:obstructions}
 c d^{-k}g_n^{\eps},\quad g_n^{\eps}d^kc^{-1}.
\end{equation} 
Now suppose the label $v$ of a subpath of $C_n$ is a piece. It follows that $v$ does not contain any of the above words as strict subwords; hence it must be one of the words in (\ref{eqn:obstructions}) or its inverse, a subword of $d^{-2f(n)}g_nd^{2f(n)}$, or a (strict) subword of a word of the form $g_n^\eps w g_n^{\eps'}$ where $w\in F(c,d)$ has length at most $6f(n)+2$. Note that each word in (\ref{eqn:obstructions}) has length at most $2f(n)+1+|g_n|$ and that that $C_n$ has length at least $f(n)(3f(n)+4|g_n|)$, as each commutator in the product contributes 3 copies of $d^{f(n)}$ and 4 copies of $g_n$ to a reduced decomposition of $R_n$. It follows that $R$ satisfies $C'(\frac16)$, since
\[
 6(6f(n)+2+2|g_n|)\leq 36f(n) + 12 + 12|g_n| < f(n)(3f(n)+4|g_n|).
\]
\end{example}

For each $n$ set $S_n=\setcon{d^{i}cd^{-i}g_nd^{i}c^{-1}d^{-i}}{1\leq i\leq 2f(n)}$. Our next goal is the following.

\begin{proposition} \label{prop:surfacesubgroup} For each $n$, the subgroup of $G$ generated by $S_n$ is isomorphic to the fundamental group of a closed surface of genus $f(n)$.
\end{proposition}
It suffices to prove that the $F(S_n)\cap \ngen{\set{R_k}}^{F(S)} = \ngen{R_n}^{F(S_n)}$, where $\ngen{X}^G$ denotes the subgroup of $G$ normally generated by $X$. In particular, we need to show that if $w\in F(S_n)$ satisfies $w=_G 1$ then $w$ is equal in $F(S)$ to a product of conjugates of $R_n$ by elements of $F(S_n)$.

\begin{lemma}\label{lem:reducedform} Any word in $F(S_n)$ has a reduced decomposition in $F(S)$ which is obtained from a word of the following form only by cancelling words of the form $dd^{-1}$ and $d^{-1}d$:
\begin{equation}\label{eqn:Sreduced}
 w =_{F(S)} \Pi_{l=1}^m d^{x_l}cd^{-x_l} g_n^{y_l} d^{x_l}c^{-1}d^{-x_l}
\end{equation}
where each $x_l>0$, $y_l\in\Z\setminus\set{0}$ and $x_l\neq x_{l+1}$.
\end{lemma}
\begin{proof}
This obviously holds, since different elements of $S_i$ all begin with different powers of $d$.
\end{proof}

\begin{proof}[Proof of Proposition \ref{prop:surfacesubgroup}]
Let $w\in F(S_n)$ satisfy $w=_G 1$, and let $D_w$ be a minimal diagram whose boundary label read clockwise from a fixed vertex $e$ is the representation of $w$ in the free monoid $M(S\sqcup S^{-1})$ presented in (\ref{eqn:Sreduced}). Note that this diagram has vertices of degree $1$ (which we will call \textbf{spurs}).  Such a vertex is always connected to a vertex of degree at least $3$ in $D_w$ by a unique path, all of whose internal vertices have degree $2$. Necessarily the label of this path is a power of $d$; this corresponds exactly to the fact that the representation of $w$ in  (\ref{eqn:Sreduced}) can be reduced to a cyclically reduced word in $F(S)$ only by cancelling words of the form $dd^{-1}$ and $d^{-1}d$.

By Greendlinger's lemma, there is a face $\Pi$ in $D_w$ such that $\partial \Pi\cap \partial D_w$ is connected and has at least $\frac12\abs{\partial\Pi}$ edges (it may be the whole of $\partial \Pi$). Then the label of $\partial\Pi$ is equal in $F(S)$ to some $R_k$, and, in particular, some subpath of $\partial D_w$ has label $dg_k^{\eps}d^{-1}$ for $\eps\in\set{\pm 1}$. Since all of the $g_i$ are distinct, and the label of $\partial D_w$ does not contain a subpath labeled $dg_id^{-1}$ for any $i\neq n$, we deduce that $k=n$.

Moreover, there is a spur $x$ in $D_w$ which is connected by a path to a  vertex in $\partial \Pi$ of degree at least $3$. We see immediately that the label of the subpath of $\partial D_w$ from $e$ to $x$ has exactly the form given in (\ref{eqn:Sreduced}) and consequently is an element of $F(S_n)$.  It also follows that the label of $\partial\Pi$ read from $x$ is (a reduction of) a cyclic conjugate of $R_n$ by an element of $F(S_n)$.  Consequently, $w$ can be expressed as the product of a word $w'\in F(S_n)$ satisfying $w'=_G 1$ and a conjugate of $R_n^{\pm 1}$ by an element of $F(S_n)$.  Moreover, we ensure that $l_{F(S)}(w')<l_{F(S)}(w)$.  Thus by induction on $l_{F(S)}(w)$ we see that any $w\in F(S_n)$ with the property that $w=_G 1$, can be expressed as a product of conjugates of $R_n^{\pm 1}$ by words in $F(S_n)$.
\end{proof}

We illustrate the above proof in the figure below.
\begin{figure}[H]
 \centering

\begin{tikzpicture}[xscale=0.8, yscale=0.8, vertex/.style={draw,fill,circle,inner sep=0.3mm},middlearrow/.style 2 args={
        decoration={             
            markings, 
            mark=at position 0.5 with {\arrow[xshift=3.333pt]{triangle 45}, \node[#1] {#2};}
        },
        postaction={decorate}
    }]

\filldraw[fill=black!10!white]
	(3,0) -- (3,3) -- (9,3) -- (9,0) -- (3,0);

\node[vertex]
(1) at (0,0) {};
\path (1,0) node[below] {$...cd^{-x_1}g_nd^{x_1}c^{-1}...$};

\node[vertex]
(2) at (2,0) {};

\node[vertex]
(3) at (3,0) {};

\node[vertex]
(4) at (3,1) {};

\draw[] (3,1) -- (2,1);

\node[vertex]
(4b) at (2,1) {};
\path (2,1) node[left] {$x$};

\node[vertex]
(5) at (3,1) {};

\node[vertex]
(6) at (3,3) {};
\path (3,2) node[right] {$...cd^{-x_k}g_n^{\pm 1}d^{x_k}c^{-1}...$};

\node[vertex]
(6b) at (4,4) {};

\node[vertex]
(10) at (9,0) {};

\node[vertex]
(14) at (12,0) {};

\draw[middlearrow={below}{}]
	(1) -- (2);
\draw[]
	(3) -- (2);
\draw[]
	(4) -- (3);
\draw[]
	(4) -- (5);
\draw[middlearrow={below}{}]
	(5) -- (6);
\draw[]
	(10) -- (14);
\draw[] (4,4)--(4,3);

\path (6,1) node[] {$\Pi$};

\draw[dashed] (-0.8,1)-- (-0.8,0.2) -- (2.8,0.2) -- (2.8,0.8) -- (1.5,0.8)-- (1.5,1.2) -- (2.8,1.2) --(2.8,3.2) -- (3.8,3.2) -- (3.8,4.2)-- (4.2,4.2)-- (4.2,3.2) -- (9.2,3.2) -- (9.2,0.2) -- (12.2,0.2);

\path (1,0.2) node[above] {$\partial D$};

\node[vertex]
(13) at (-1,1) {};
\path (-1,1) node[above] {$e$};
\node[vertex]
(14) at (-1,0) {};

\draw 	(-1.2,0) -- (-0.8,0);
\draw[middlearrow={below}{}] (-1,1)--(-1,0);
\draw[dashed] (-0.8,0)--(0,0);

\path (-1,0.5) node[left] {$d^{x_1-x_m}$};

\end{tikzpicture}
 \caption{Intersection of relators and boundary words}
 \label{figRBP}
\end{figure}
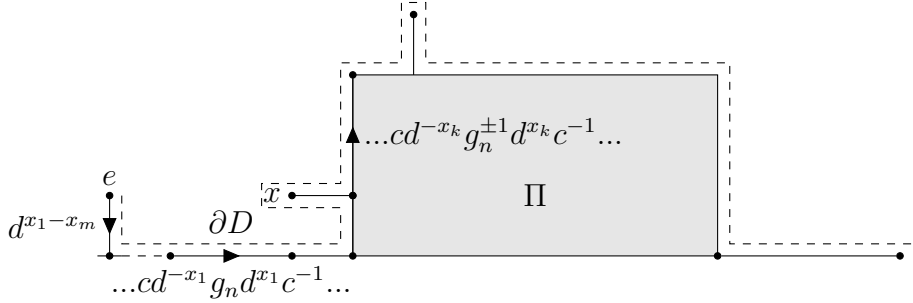

%

We now use the above construction to prove the following result, from which Theorems \ref{thm:contnotgenlox} and \ref{thm:exnotglox} immediately follow.  Given a set $R=\{r_1,r_2,\dots\}\subset F(S)$, let $C_n=C_{r_n}$, and let $p_n=p_{r_n}$ denote the length of the longest piece in $r_n$.

\begin{corollary} For every unbounded function $\rho:\N\to\N$ there exists a set $R=\{r_1,r_2,\dots\}\subset F(a,b,c,d)$ such that $p_n\cdot\sat_{b}(C_n)\preceq\rho(n)$ and $\fgen{b}$ is $\rho'$--contracting for some $\rho'\asymp\rho$, but $b$ is not generalized loxodromic.
\end{corollary}
\begin{proof} Without loss of generality, we may assume that $\rho(n)/n\to 0$ as $n\to\infty$.
In the construction in Example \ref{ex:construction}, set $g_n=a^nb^{\rho(n)}a^{-n}$ and $f(n)=n$. It is clear that $b$ has infinite order in the resulting group $G=F(a,b,c,d)/\langle R\rangle$, and by Theorem \ref{ACGHcontracting}, $\fgen{b}$ is $\rho'$--contracting for some $\rho'\asymp \rho$.

The longest piece in $C_n$ has length $\asymp n$.  To find an upper bound on the saturation notice that $C_n$ contains exactly $4n$ copies of $b^{\rho(n)}$ and that $\abs{C_n}\asymp n^2$. Therefore,
\[
p_n\cdot\sat_{b}(C_n)\preceq n\cdot\frac{4n\rho(n)}{n^2} \preceq \rho(n).
\]

Suppose for a contradiction that $b$ is a generalized loxodromic element.  Then there exists a constant $N$ such that for every $n\geq N$, $\ngen{b^n}$ is a free group.  Choose $k\in\mathbb N$ so that $\rho(k)\geq N$. By construction, $\ngen{b^{\rho(k)}}$ contains the subgroup of $G$ generated by $S_k$.  However, this subgroup is not free by Proposition \ref{prop:surfacesubgroup}, yielding the contradiction.
\end{proof}

We are now ready to prove Theorem \ref{thm:tfnonUA}.

\begin{proof}[Proof of Theorem \ref{thm:tfnonUA}] Let $w^*=(w_k)_{k\in\N}=abbabaab...$ be the Thue-Morse sequence, that is, the infinite word generated by iterating the procedures $a\mapsto ab$ and $b\mapsto ba$ on the word $a$. This sequence is famously triple-free.
For each $n$ let
\[
 g_n = (b^3w_{2^n}w_{2^n+1}\ldots w_{2^{n+1}-1})^n
\] in the construction in Example \ref{ex:construction}, and let $G=F(a,b,c,d)/\langle R\rangle$ be the resulting group.

We will first show that every non-trivial element of $F(a,b)$ is generalized loxodromic in $G$.  Suppose for a contradiction that $v$ is a cyclically reduced word in $F(a,b)$ such that arbitrary powers of $v$ appear as labels of subpaths of the cycles $C_n$.  By construction of $R$, $v$ must be a subword of $g_n$ for some $n$.  Choose a cyclic conjugate of $v$ of the form $b^kv'$ with $\abs{k}$ maximal. If $|k|\geq 3$,  
then for all $n$ sufficiently large, no (non-zero) power of $v$ is a subword of $g_n$, which is a contradiction.  If $|k|<3$, 
then for every $j>0$ there is some $n>0$ such that $v^j$ is a subword of $b^2w_{2^n+1}\ldots w_{2^{n+1}-1}b^2$.  However, the Thue-Morse sequence is triple-free, so this word contains no non-trivial seventh powers, and we again reach a contradiction.  Thus $v$ is $D$--contracting for some constant $D$, and so by \cite[Theorem 6.2]{ACGH2}, $v$ is a generalized loxodromic element.  Therefore, every non-trivial element of $F(a,b)$ is generalized loxodromic in $G$. 

We next show that $G$ is not universally acylindrical.  Suppose for a contradiction that $G$ admits a universal acylindrical action. Then by \cite[Theorem 5.3 and Proposition 6.34(b)]{DGO17} there exists a uniform constant $m$ such that for every generalized loxodromic element $h\in G$, $\ngen{h^m}$ is a free subgroup of $G$.  Since every non-trivial element of $F(a,b)$ is generalized loxodromic, it follows that for every $n\geq m$, $\ngen{g_n}$ is a free subgroup of $G$ .   However, by construction, $\ngen{g_n}$ contains the subgroup of $G$ generated by $S_n$ which by Proposition \ref{prop:surfacesubgroup} is not free, yielding the contradiction.

The same proof will work for any infinite subset of $\set{R_n}$. Using Bowditch's taut loop spectrum \cite{BowditchUncQI} as an invariant we obtain non-quasi-isometric groups which are not universally acylindrical.
\end{proof}

\begin{remark}
By adjusting the above construction, it is possible to give an example of a finitely generated group in which every non-trivial element is generalized loxodromic but which is not universally acylindrical.  For example, one could replace the powers of $d$ in (\ref{eqn:defineRn}) by a collection of words in $F(c,d)$ which are uniformly power-free.
\end{remark}

%% file: qgeo1.pdf_tex
\begingroup%
  \makeatletter%
  \providecommand\color[2][]{%
    \errmessage{(Inkscape) Color is used for the text in Inkscape, but the package 'color.sty' is not loaded}%
    \renewcommand\color[2][]{}%
  }%
  \providecommand\transparent[1]{%
    \errmessage{(Inkscape) Transparency is used (non-zero) for the text in Inkscape, but the package 'transparent.sty' is not loaded}%
    \renewcommand\transparent[1]{}%
  }%
  \providecommand\rotatebox[2]{#2}%
  \ifx\svgwidth\undefined%
    \setlength{\unitlength}{458.74081349bp}%
    \ifx\svgscale\undefined%
      \relax%
    \else%
      \setlength{\unitlength}{\unitlength * \real{\svgscale}}%
    \fi%
  \else%
    \setlength{\unitlength}{\svgwidth}%
  \fi%
  \global\let\svgwidth\undefined%
  \global\let\svgscale\undefined%
  \makeatother%
  \begin{picture}(1,0.42394685)%
    \put(0,0){\includegraphics[width=\unitlength,page=1]{qgeo1.pdf}}%
    \put(-0.00185604,0.02625448){\color[rgb]{0,0,0}\makebox(0,0)[lb]{\smash{$1$}}}%
    \put(0.83753081,0.07211177){\color[rgb]{0,0,0}\makebox(0,0)[lb]{\smash{$w^{m(n)}$}}}%
    \put(0.42880321,0.40507513){\color[rgb]{0,0,0}\makebox(0,0)[lb]{\smash{$x$}}}%
    \put(0.22543631,0.22363996){\color[rgb]{0,0,0}\makebox(0,0)[lb]{\smash{$y$}}}%
    \put(0.21147982,0.04220485){\color[rgb]{0,0,0}\makebox(0,0)[lb]{\smash{$z$}}}%
    \put(0.56637491,0.04021106){\color[rgb]{0,0,0}\makebox(0,0)[lb]{\smash{$z'$}}}%
    \put(0.65011424,0.25952825){\color[rgb]{0,0,0}\makebox(0,0)[lb]{\smash{$y'$}}}%
    \put(0.74182873,0.17578892){\color[rgb]{0,0,0}\makebox(0,0)[lb]{\smash{$\alpha_n$}}}%
    \put(0.72189072,0.00432277){\color[rgb]{0,0,0}\makebox(0,0)[lb]{\smash{$P_n$}}}%
    \put(0.24736806,0.13990063){\color[rgb]{0,0,0}\makebox(0,0)[lb]{\smash{$\gamma$}}}%
    \put(0.62220106,0.12793793){\color[rgb]{0,0,0}\makebox(0,0)[lb]{\smash{$\gamma'$}}}%
  \end{picture}%
\endgroup%